\documentclass{amsart}
\usepackage{amsfonts}
\usepackage{cite}
\usepackage{hyperref}

\newtheorem{theorem}{Theorem}
\theoremstyle{plain}

\newtheorem{corollary}{Corollary}

\newtheorem{remark}{Remark}

\numberwithin{equation}{section}

\begin{document}
\title{The $\mathcal{W}^{\ast }-$curvature tensor on relativistic space-times%
}
\author{H. M. Abu-Donia}
\address{Department of Mathematics, Faculty of Science, Zagazig \\
University, Egypt,}
\email{donia\_1000@yahoo.com}
\author{Sameh Shenawy}
\address[S. Shenawy]{Basic Science Department, Modern Academy for
Engineering and Technology, Maadi, Egypt,}
\email{drshenawy@mail.com}
\author{Abdullah A. Syied}
\address[A. A. Syied]{ Department of Mathematics, Faculty of Science,
Zagazig \\
University, Egypt,}
\email{a.a\_syied@yahoo.com}
\subjclass[2010]{Primary 53C25, 53C50; Secondary 53C80, 53B20}
\keywords{Einstein's field equation, Perfect fluid space-times,
Energy-momentum tensor, semi-symmetric curvature tensor.}

\begin{abstract}
This paper aims to study the $\mathcal{W}^{\ast }-$curvature tensor on
relativistic space-times. The energy-momentum tensor $T$ of a space-time is
semi-symmetric given that the $\mathcal{W}^{\star }-$curvature tensor is
semi-symmetric whereas energy-momentum tensor $T$ of a space-time having a
divergence free $\mathcal{W}^{\star }-$curvature tensor is of Codazzi type.
A space-time having a traceless $\mathcal{W}^{\ast }-$curvature tensor is
Einstein. A $\mathcal{W}^{\ast }-$curvature flat space-time is Einstein.
Perfect fluid space-times which admits $\mathcal{W}^{\star }-$curvature
tensor are considered.
\end{abstract}

\maketitle

\section{Introduction}

In \cite%
{Pokhariyal:1970,Pokhariyal:1971,Pokhariyal:1972,Pokhariyal:1982,Pokhariyal:2001}%
, the authors introduced some curvature tensors similar to the projective
curvature tensor\cite{Mishra:1984}. They investigated their geometrical
properties and physical significance. These tensors have been recently
studied in different ambient spaces \cite%
{Ahsan:2017,Mallick:2014,Mallick:2016,Shenawy:2016,sach:1977,Taleshian:2010,Ozen:2011}%
. However, we noticed that a little attention is paid to the $\mathcal{W}%
_{3}^{\ast }-$curvature tensor. This tensor is a $\left( 0,4\right) $ tensor
defined as%
\begin{equation*}
\mathcal{W}_{3}^{\star }\left( U,V,Z,T\right) =R\left( U,V,Z,T\right) -\frac{%
1}{n-1}\left[ g\left( V,Z\right) \mathrm{Ric}\left( U,T\right) -g\left(
V,T\right) \mathrm{Ric}\left( U,Z\right) \right] ,
\end{equation*}%
where $R\left( U,V,Z,T\right) =g(R\left( \left( U,V\right) Z,T\right) $, $%
R\left( U,V\right) Z=\nabla _{U}\nabla _{V}-\nabla _{U}\nabla _{V}-\nabla
_{\lbrack U,V]}Z$ is the Riemann curvature tensor, $\nabla $ is the
Levi-Civita connection, and $\mathrm{Ric}\left( U,V\right) $ is Ricci
tensor. For the simplicity, we will denote $\mathcal{W}_{3}^{\star }$ by $%
\mathcal{W}^{\star }$. In the local coordinates, it is%
\begin{equation}
\mathcal{W}_{ijkl}^{\star }=R_{ijkl}-\frac{1}{n-1}\left[
g_{jk}R_{il}-g_{jl}R_{ik}\right] .  \label{F2}
\end{equation}%
The $\mathcal{W}^{\ast }-$curvature tensor does not have neither symmetry
nor cyclic properties.

A semi-Riemannian manifold $M$ is semi-symmetric \cite{Szabo:2016} if%
\begin{equation*}
R\left( \zeta ,\xi \right) \cdot R=0,
\end{equation*}%
where $R\left( \zeta ,\xi \right) $ acts as a derivation on $R$. $M$ is
Ricci semi-symmetric \cite{Mirzoyan:1991} if%
\begin{equation*}
R\left( \zeta ,\xi \right) \cdot \mathrm{Ric}=0,
\end{equation*}%
where $R\left( \zeta ,\xi \right) $ acts as a derivation on $\mathrm{Ric}$.
A semi-symmetric manifold is known to be Ricci semi-symmetric as well. The
converse does not generally hold. On the same line of the above definitions
we say that $M$ has a semi-symmetric $\mathcal{W}^{\star }-$curvature tensor
if%
\begin{equation*}
R\left( \zeta ,\xi \right) \cdot \mathcal{W}^{\star }=0,
\end{equation*}%
where $R\left( \zeta ,\xi \right) $ acts as a derivation on $\mathcal{W}%
^{\star }$.

This study was designed to fill this observed gap. The relativistic
significance of the $\mathcal{W}^{\star }-$curvature tensor is investigated.
First, it is shown that space-times with semi-symmetric $\mathcal{W}%
_{jk}^{\star }=g^{il}\mathcal{W}_{ijkl}^{\star }$ tensor have Ricci
semi-symmetric tensor and consequently the energy-momentum tensor is
semi-symmetric. The divergence of the $\mathcal{W}^{\star }-$curvature
tensor is considered and it is proved that the energy-momentum tensor $T$ of
a space-time $M$ is of Codazzi type if $M$ has a divergence free $\mathcal{W}%
^{\star }-$curvature tensor. If $M$ admits a parallel $\mathcal{W}^{\star }-$%
curvature tensor, then $T$ is a parallel. Finally, a $\mathcal{W}^{\star }-$%
flat perfect fluid space-time performs as a cosmological constant. A dust
fluid\emph{\ }$\mathcal{W}^{\star }-$flat space-time satisfies Einstein's
field equation is a vacuum space.

\section{$\mathcal{W}^{\star }-$semi-symmetric space-times}

A $4-$dimensional relativistic space-time $M$ is said to have a
semi-symmetric $\mathcal{W}^{\star }-$curvature tensor if%
\begin{equation*}
R\left( \zeta ,\xi \right) \cdot \mathcal{W}^{\star }=0,
\end{equation*}%
where $R\left( \zeta ,\xi \right) $ acts as a derivation on the tensor $%
\mathcal{W}^{\star }$. In local coordinates, one gets

\begin{eqnarray}
\left( \nabla _{\mu }\nabla _{\nu }-\nabla _{\nu }\nabla _{\mu }\right) 
\mathcal{W}_{ijkl}^{\star } &=&\left( \nabla _{\mu }\nabla _{\nu }-\nabla
_{\nu }\nabla _{\mu }\right) R_{ijkl}-\frac{1}{3}[g_{jk}\left( \nabla _{\mu
}\nabla _{\nu }-\nabla _{\nu }\nabla _{\mu }\right) R_{il}  \notag \\
&&-g_{jl}\left( \nabla _{\mu }\nabla _{\nu }-\nabla _{\nu }\nabla _{\mu
}\right) R_{ik}.  \label{1F}
\end{eqnarray}%
Contracting both sides with $g^{il}$ yields%
\begin{equation}
\left( \nabla _{\mu }\nabla _{\nu }-\nabla _{\nu }\nabla _{\mu }\right) 
\mathcal{W}_{jk}^{\star }=\frac{4}{3}\left( \nabla _{\mu }\nabla _{\nu
}-\nabla _{\nu }\nabla _{\mu }\right) R_{jk},
\end{equation}%
where $\mathcal{W}_{jk}^{\star }=g^{il}\mathcal{W}_{ijkl}^{\star }$. Thus we
have the following theorem.

\begin{theorem}
$M$ is Ricci semi-symmetric if and only if $\mathcal{W}_{jk}^{\star }=g^{il}%
\mathcal{W}_{ijkl}^{\star }$ is semi-symmetric.
\end{theorem}

The following result is a direct consequence of this theorem.

\begin{corollary}
$M$ is Ricci semi-symmetric if the $\mathcal{W}^{\ast }-$curvature is
semi-symmetric.
\end{corollary}

A space-time manifold is conformally semi-symmetric if the conformal
curvature tensor $\mathcal{C}$ is semi-symmetric.

\begin{theorem}
Assume that $M$ is a space-time admitting a semi-symmetric $\mathcal{W}%
_{jk}^{\star }=g^{il}\mathcal{W}_{ijkl}^{\star }$. Then, $M$ is conformally
semi-symmetric if and only if it is semi-symmetric i.e. $\nabla _{\lbrack
\mu }\nabla _{\nu ]}R_{ijkl}=0\Leftrightarrow \nabla _{\lbrack \mu }\nabla
_{\nu ]}\mathcal{C}_{ijkl}=0$.
\end{theorem}

The Einstein's field equation is%
\begin{equation}
R_{ij}-\frac{1}{2}g_{ij}R+g_{ij}\Lambda =kT_{ij},  \label{4-F}
\end{equation}%
where $\Lambda ,R,k$ are the cosmological constant, the scalar curvature,
and the gravitational constant. Then%
\begin{equation}
\left( \nabla _{\mu }\nabla _{\nu }-\nabla _{\nu }\nabla _{\mu }\right)
R_{ij}=k\left( \nabla _{\mu }\nabla _{\nu }-\nabla _{\nu }\nabla _{\mu
}\right) T_{ij},
\end{equation}%
i.e., $M$ is Ricci semi-symmetric if and only if the energy-momentum tensor
is semi-symmetric.

\begin{theorem}
The energy-momentum tensor of a space-time $M$ is semi-symmetric if and only
if $\mathcal{W}_{jk}^{\star }=g^{il}\mathcal{W}_{ijkl}^{\star }$ is
semi-symmetric.
\end{theorem}

\begin{remark}
A space-time $M$ with semi-symmetric energy-momentum tensor has been studied
by De and Velimirovic in \cite{De:2014}.
\end{remark}

It is clear that $\nabla _{\mu }\mathcal{W}_{ijkl}^{\star }=0$ implies $%
\left( \nabla _{\mu }\nabla _{\nu }-\nabla _{\nu }\nabla _{\mu }\right) 
\mathcal{W}_{ijkl}^{\star }=0$. Thus the following result rises.

\begin{corollary}
Let $M$ be a space-time having a covariantly constant $\mathcal{W}^{\ast }-$%
curvature tensor. Then $M$ is conformally semi-symmetric and the
energy-momentum tensor is semi-symmetric.
\end{corollary}

A space-time is called Ricci recurrent if the Ricci curvature tensor
satisfies%
\begin{equation}
\nabla _{\mu }R_{ij}=b_{\mu }R_{ij},
\end{equation}%
where $b$ is called the associated recurrence $1-$form. Assume that the
Ricci tensor is recurrent, then%
\begin{eqnarray}
\left( \nabla _{\mu }\nabla _{\nu }-\nabla _{\nu }\nabla _{\mu }\right)
R_{ij} &=&\nabla _{\mu }\left( \nabla _{\nu }R_{ij}\right) -\nabla _{\nu
}\left( \nabla _{\mu }R_{ij}\right)  \notag \\
&=&\nabla _{\mu }\left( b_{\nu }R_{ij}\right) -\nabla _{\nu }\left( b_{\mu
}R_{ij}\right)  \notag \\
&=&\left( \nabla _{\mu }b_{\nu }\right) R_{ij}+b_{\nu }\nabla _{\mu
}R_{ij}-\left( \nabla _{\nu }b_{\mu }\right) R_{ij}-b_{\mu }\nabla _{\nu
}R_{ij}  \notag \\
&=&\left[ \nabla _{\mu }b_{\nu }-\nabla _{\nu }b_{\mu }\right] R_{ij}.
\end{eqnarray}

\begin{corollary}
The following conditions on a space-time $M$ are equivalent

\begin{enumerate}
\item The Ricci tensor is recurrent with closed recurrence one form,

\item $T$ is semi-symmetric, and

\item $\mathcal{W}_{jk}^{\star }=g^{il}\mathcal{W}_{ijkl}^{\star }$ is
semi-symmetric.
\end{enumerate}
\end{corollary}

\section{Space-times admitting divergence free $\mathcal{W}^{\star }-$%
curvature tensor}

The tensor $\mathcal{W}_{jkl}^{\star h}$ of type $\left( 1,3\right) $ is
given by%
\begin{eqnarray*}
\mathcal{W}_{jkl}^{\star h} &=&g^{hi}\mathcal{W}_{ijkl}^{\star } \\
&=&R_{jkl}^{h}-\frac{1}{3}[g_{jk}R_{l}^{h}-g_{jl}R_{k}^{h}].
\end{eqnarray*}%
Consequently, one defines its divergence as%
\begin{eqnarray}
\nabla _{h}\mathcal{W}_{jkl}^{\star h} &=&\nabla _{h}R_{jkl}^{h}-\frac{1}{3}%
[g_{jk}\nabla _{h}R_{l}^{h}-g_{jl}\nabla _{h}R_{k}^{h}]  \notag \\
&=&\nabla _{h}R_{jkl}^{h}-\frac{1}{3}[g_{jk}\nabla _{l}R-g_{jl}\nabla _{k}R].
\label{F240}
\end{eqnarray}%
It is\ well known that the contraction of the second Bianchi identity gives%
\begin{equation*}
\nabla _{h}R_{jkl}^{h}=\nabla _{l}R_{jk}-\nabla _{k}R_{jl}.
\end{equation*}%
Thus, Equation (\ref{F240}) becomes 
\begin{equation}
\nabla _{h}\mathcal{W}_{jkl}^{\star h}=\nabla _{l}R_{jk}-\nabla _{k}R_{jl}-%
\frac{1}{3}[g_{jk}\nabla _{l}R-g_{jl}\nabla _{k}R].  \label{1-F24}
\end{equation}%
If the $\mathcal{W}^{\star }-$curvature tensor is divergence free, then
Equation (\ref{1-F24}) turns into 
\begin{equation*}
0=\nabla _{l}R_{jk}-\nabla _{k}R_{jl}-\frac{1}{3}[g_{jk}\nabla
_{l}R-g_{jl}\nabla _{k}R].
\end{equation*}%
Multiplying by $g^{jk}$ we have 
\begin{equation}
\nabla _{l}R=0.  \label{F26}
\end{equation}%
Thus, the tensor $R_{ij}$ is a Codazzi tensor and $R$ is constant.
Conversely, assume that the Ricci tensor is a Codazzi tensor. Then%
\begin{eqnarray*}
\nabla _{h}\mathcal{W}_{jkl}^{\star h} &=&-\frac{1}{3}[g_{jk}\nabla
_{l}R-g_{jl}\nabla _{k}R] \\
0 &=&\nabla _{l}R_{jk}-\nabla _{k}R_{jl}
\end{eqnarray*}%
However, the last equation implies that $\nabla _{l}R=0$. Consequently, the $%
\mathcal{W}^{\star }-$curvature tensor has zero divergence.

\begin{theorem}
The $\mathcal{W}^{\star }-$curvature tensor has zero divergence if and only
if the Ricci tensor is a Codazzi tensor. In both cases, the scalar curvature
is constant.
\end{theorem}

The divergence of the Weyl curvature $\mathcal{C}$ tensor is given by%
\begin{equation*}
\nabla _{h}\mathcal{C}_{ijk}^{h}=\frac{n-3}{n-2}\left[ \nabla
_{k}R_{ij}-\nabla _{j}R_{ik}\right] +\frac{1}{2\left( n-1\right) }%
[g_{ij}\nabla _{k}R-g_{ik}\nabla _{j}R].
\end{equation*}

\begin{remark}
Since divergence free of $\mathcal{W}^{\star }-$curvature tensor implies
that $R_{ij}$ is a Codazzi tensor, the conformal curvature tensor has zero
divergence.
\end{remark}

Equation (\ref{4-F}) yields%
\begin{equation*}
\nabla _{l}R_{ij}-\frac{1}{2}g_{ij}\nabla _{l}R=k\nabla _{l}T_{ij}.
\end{equation*}%
The above theorem now implies the following result.

\begin{corollary}
The energy-momentum tensor is a Codazzi tensor if and only if the $\mathcal{W%
}^{\star }-$curvature tensor has zero divergence. In both cases, the scalar
curvature is constant.
\end{corollary}

Einstein's field equation infers

\begin{eqnarray}
k\left( \nabla _{l}T_{ij}-\nabla _{i}T_{jl}\right) &=&\nabla _{l}\left(
R_{ij}-\frac{1}{2}g_{ij}R\right) -\nabla _{i}\left( R_{lj}-\frac{1}{2}%
g_{lj}R\right)  \label{F29} \\
&=&\nabla _{l}R_{ij}-\nabla _{i}R_{lj}-\frac{1}{2}\left( g_{ij}\nabla
_{l}R-g_{lj}\nabla _{i}R\right) \\
&=&\nabla _{h}\mathcal{W}_{jil}^{\star h}-\frac{1}{6}\left( g_{ij}\nabla
_{l}R-g_{lj}\nabla _{i}R\right) .
\end{eqnarray}%
Now, it is noted that the above theorem may be proved using this identity.

\section{$\mathcal{W}^{\star }-$symmetric space-times}

A space-time $M$ is called $\mathcal{W}^{\star }-$symmetric if%
\begin{equation*}
\nabla _{m}\mathcal{W}_{ijkl}^{\star }=0.
\end{equation*}%
Applying the covariant derivative on the both sides of equation (\ref{F2}),
one gets%
\begin{equation}
\nabla _{m}\mathcal{W}_{ijkl}^{\star }=\nabla _{m}R_{ijkl}-\frac{1}{n-1}%
\left[ g_{jk}\nabla _{m}R_{il}-g_{jl}\nabla _{m}R_{ik}\right] .  \label{F4}
\end{equation}%
If $M$ is a $\mathcal{W}^{\star }-$symmetric space-time, then%
\begin{equation*}
\nabla _{m}R_{ijkl}=\frac{1}{3}[g_{jk}\nabla _{m}R_{il}-g_{jl}\nabla
_{m}R_{ik}].
\end{equation*}%
Multiplying the both sides by $g^{il},$ we get%
\begin{equation*}
\nabla _{m}R_{jk}=\frac{1}{3}[g_{jk}\nabla _{m}R-\nabla _{m}R_{jk}],
\end{equation*}%
and hence%
\begin{equation}
\nabla _{m}R_{jk}=\frac{1}{4}g_{jk}\nabla _{m}R.  \label{F5}
\end{equation}

Now, the following theorem rises.

\begin{theorem}
Assume that $M$ is a $\mathcal{W}^{\star }-$symmetric space-time, then $M$
is a Ricci symmetric if the scalar curvature is constant.
\end{theorem}

The second Bianchi identity for $\mathcal{W}^{\star }-$curvature tensor is%
\begin{eqnarray}
\nabla _{m}\mathcal{W}_{ijkl}^{\star }+\nabla _{k}\mathcal{W}_{ijlm}^{\star
}+\nabla _{l}\mathcal{W}_{ijmk}^{\star } &=&-\frac{1}{3}\left[ g_{jk}(\nabla
_{m}R_{il}-\nabla _{l}R_{im})+g_{jl}(\nabla _{k}R_{im}-\nabla _{m}R_{ik})%
\right]  \notag \\
&&-\frac{1}{3}g_{jm}(\nabla _{l}R_{ik}-\nabla _{k}R_{il}).  \label{F7}
\end{eqnarray}%
If the Ricci tensor satisfies $\nabla _{m}R_{il}=\nabla _{l}R_{im}$, then%
\begin{equation}
\nabla _{m}\mathcal{W}_{ijkl}^{\star }+\nabla _{k}\mathcal{W}_{ijlm}^{\star
}+\nabla _{l}\mathcal{W}_{ijmk}^{\star }=0.  \label{F8}
\end{equation}

Conversely, if the above equation holds, then Equation (\ref{F7}) implies%
\begin{equation}
g_{jk}(\nabla _{m}R_{il}-\nabla _{l}R_{im})+g_{jl}(\nabla _{k}R_{im}-\nabla
_{m}R_{ik})+g_{jm}(\nabla _{l}R_{ik}-\nabla _{k}R_{il})=0.  \label{1-F7}
\end{equation}%
Multiplying the both sides with $g^{ik}$, then we have 
\begin{equation}
\nabla _{m}R_{jl}=\nabla _{l}R_{jm},  \label{2-F7}
\end{equation}%
which means that the Ricci tensor is of Codazzi type.

\begin{theorem}
The Ricci tensor satisfies $\nabla _{m}R_{il}=\nabla _{l}R_{im}$ if and only
if the $\mathcal{W}^{\star }-$curvature tensor satisfies Equation (\ref{F8}).
\end{theorem}

For a purely electro-magnetic distribution, Equation (\ref{4-F}) reduces to 
\begin{equation}
R_{ij}=kT_{ij}.  \label{F2-30}
\end{equation}%
Its contraction with $g^{_{ij}}$ gives 
\begin{equation}
R=-kT.  \label{F3-30}
\end{equation}%
In this case, it is $T=R=0$. Thus Equation (\ref{F5}) yields $\nabla
_{m}T_{jk}=0.$

\begin{theorem}
The energy-momentum tensor of a $\mathcal{W}^{\star }-$symmetric space-time
obeying Einstein's field equation for a purely electro-magnetic distribution
is locally symmetric.
\end{theorem}

\section{$\mathcal{W}^{\star }-$flat space-times}

Now, we consider $\mathcal{W}^{\star }-$flat space-times. Multiplying both
sides of Equation (\ref{F2}) by $g^{il}$ yields%
\begin{eqnarray*}
\mathcal{W}_{jk}^{\star } &=&g^{il}\mathcal{W}_{ijkl}^{\star } \\
&=&\frac{4}{3}\left( R_{jk}-\frac{R}{4}g_{jk}\right) .
\end{eqnarray*}%
Thus, a $\mathcal{W}_{jk}^{\star }-$curvature flat space-time is Einstein,
i.e.,%
\begin{equation}
R_{jk}=\frac{R}{4}g_{jk}.  \label{F-3}
\end{equation}%
Now, Equation (\ref{F2}) becomes%
\begin{equation*}
\mathcal{W}_{ijkl}^{\star }=R_{ijkl}-\frac{R}{12}[g_{ik}g_{jl}-g_{jl}g_{jk}].
\end{equation*}

\begin{theorem}
A space-time manifold $M$ is Einstein if and only if $\mathcal{W}%
_{jk}^{\star }=0$. Moreover, a $\mathcal{W}^{\ast }-$flat space-time has a
constant curvature.
\end{theorem}

A vector field $\xi $ is said to be a conformal vector field if%
\begin{equation*}
\mathcal{L}_{\xi }g=2\phi g,
\end{equation*}%
where $\mathcal{L}_{\xi }$ denotes the Lie derivative along the flow lines
of $\xi $ and $\phi $ is a scalar. $\xi $ is called Killing if $\phi =0$.
Let $T_{ij}$ be the energy-momentum tensor defined on $M$. $\xi $ is said to
be a matter inheritance collineation if%
\begin{equation*}
\mathcal{L}_{\xi }T=2\phi T.
\end{equation*}%
The tensor $T_{ij}$ is said to have a symmetry inheritance property along
the flow lines of $\xi $. $\xi $ is called a matter collineation if $\phi =0$%
. A Killing vector field $\xi $ is a matter collineation. However, a matter
collineation is not generally Killing.

\begin{theorem}
Assume that $M$ is a $\mathcal{W}^{\star }-$flat space-time. Then, $\xi $ is
conformal if and only if $\mathcal{L}_{\xi }T=2\phi T$.
\end{theorem}

\begin{proof}
Using Equations (\ref{F-3}) and (\ref{4-F}), we have%
\begin{equation}
\left( \Lambda -\frac{R}{4}\right) g_{ij}=kT_{ij}.  \label{F-6}
\end{equation}%
Then%
\begin{equation}
\left( \Lambda -\frac{R}{4}\right) \mathcal{L}_{\xi }g=k\mathcal{L}_{\xi }T.
\label{F-7}
\end{equation}%
Assume that $\xi $ is conformal. The above two equations lead to%
\begin{eqnarray*}
2\phi \left( \Lambda -\frac{R}{4}\right) g &=&k\mathcal{L}_{\xi }T \\
2\phi T &=&\mathcal{L}_{\xi }T.
\end{eqnarray*}%
Conversely, suppose that\ the energy-momentum tensor has a symmetry
inheritance property along $\xi $. It is easy to show that $\xi $ is a
conformal vector field.
\end{proof}

\begin{corollary}
Assume that $M$ is a $\mathcal{W}^{\star }-$flat space-time. Then, $M$
admits a matter collineation $\xi $ if and only if\ $\xi $ is Killing.
\end{corollary}

Equations (\ref{F-3}) and (\ref{4-F}) imply%
\begin{equation}
\left( \Lambda -\frac{R}{4}\right) g_{ij}=kT_{ij}.  \label{1-F12}
\end{equation}%
Taking the covariant derivative of \ref{1-F12} we get%
\begin{equation}
\nabla _{l}T_{ij}=\frac{1}{k}\nabla _{l}\left( \Lambda -\frac{R}{4}\right)
g_{ij}.  \label{2-F12}
\end{equation}%
Since a $\mathcal{W}^{\star }-$curvature flat space-time has $\nabla _{l}R=0$%
, $\nabla _{l}T_{ij}=0.$

\begin{theorem}
The energy-momentum tensor of a $\mathcal{W}^{\star }-$flat space-time is
covariantly constant.
\end{theorem}

Let $M$ be a space-time and $\mathcal{W}_{klm}^{\ast i}=g^{ij}\mathcal{W}%
_{jklm}^{\ast }$ be a $\left( 1,3\right) $ curvature tensor. According to 
\cite{Krupka:1996}, there exists a unique traceless tensor $\mathcal{B}%
_{klm}^{i}$ and three unique $\left( 0,2\right) $ tensors $\mathcal{C}_{kl},$
$\mathcal{D}_{kl},$ $\mathcal{E}_{kl}$ such that 
\begin{equation*}
\mathcal{W}_{klm}^{\ast i}=\mathcal{B}_{klm}^{i}+\delta _{k}^{i}\mathcal{C}%
_{lm}+\delta _{l}^{i}\mathcal{D}_{km}+\delta _{m}^{i}\mathcal{E}_{kl}.
\end{equation*}%
All of these tensors are given by%
\begin{equation*}
\mathcal{C}_{ml}=\frac{1}{33}\left[ 10\mathcal{W}_{tml}^{\ast t}-2\left( 
\mathcal{W}_{mtl}^{\ast t}+\mathcal{W}_{lmt}^{\ast t}\right) \right] =0,
\end{equation*}

\begin{eqnarray*}
\mathcal{D}_{km} &=&\frac{1}{33}\left[ -2\left( \mathcal{W}_{tkm}^{\ast t}+%
\mathcal{W}_{mkt}^{\ast t}\right) +10\mathcal{W}_{ktm}^{\ast t}\right] \\
&=&\frac{1}{9}[R_{km}-\frac{g_{km}}{4}R],
\end{eqnarray*}%
and%
\begin{eqnarray*}
\mathcal{E}_{kl} &=&\frac{1}{33}\left[ 10\mathcal{W}_{klt}^{\ast t}-2\left( 
\mathcal{W}_{tlk}^{\ast t}+\mathcal{W}_{ltk}^{\ast t}\right) \right] \\
&=&\frac{-1}{9}\left[ R_{kl}-\frac{g_{kl}}{4}R\right] .
\end{eqnarray*}%
Assume that the $\mathcal{W}^{\ast }\mathcal{-}$curvature tensor is
traceless. Then%
\begin{equation*}
\mathcal{C}_{kl}=\mathcal{D}_{kl}=\mathcal{E}_{kl}=0,
\end{equation*}%
and consequently%
\begin{equation*}
R_{ml}=\frac{g_{ml}}{4}R.
\end{equation*}

\begin{theorem}
Assume that $M$ is a space-time admitting a traceless $\mathcal{W}^{\ast }%
\mathcal{-}$curvature tensor. Then, $M$ is an Einstein space-time.
\end{theorem}

For a perfect fluid space-time with the energy density $\mu $ and isotropic
pressure $p$, it is%
\begin{equation}
T_{ij}=\left( \mu +p\right) u_{i}u_{j}+pg_{ij},  \label{F12}
\end{equation}%
where $u_{i}$ is the velocity of the fluid flow with $g_{ij}u^{_{j}}=u_{i}$
and $u_{i}u^{i}=-1$\cite{ONeill:1983,Mantica:2012,Mantica:2014}. In \cite[%
Theorem 2.2]{De:2014}, a characterization of such space-times is given. This
result leads us to.

\begin{theorem}
Assume that the perfect fluid space-time $M$ is $\mathcal{W}^{\star }-$%
semi-symmetric. Then, $M$ is regarded as inflation and this fluid acts as a
cosmological constant. Moreover, the perfect fluid represents the
quintessence barrier.
\end{theorem}

Using Equations (\ref{F-6}), we have%
\begin{equation}
\left( \Lambda -kp-\frac{R}{4}\right) g_{ij}=k\left( \mu +p\right)
u_{i}u_{j}.  \label{F13}
\end{equation}%
Multiplying the both sides by $g^{ij}$ we get%
\begin{equation}
R=4\Lambda +k\left( \mu -3p\right) .  \label{F14}
\end{equation}

For $\mathcal{W}^{\star }-$curvature flat space-times, the scalar curvature
is constant and consequently%
\begin{equation}
\mu -3p=\mathrm{constant}.  \label{F1-14}
\end{equation}%
Again, a contraction of Equation (\ref{F13}) with $u^{i}$ leads to%
\begin{equation}
R=4\left( k\mu +\Lambda \right) .  \label{F15}
\end{equation}%
The comparison between (\ref{F14}) and (\ref{F15}) gives%
\begin{equation}
\mu +p=0,  \label{F16}
\end{equation}%
i.e., the perfect fluid performs as a cosmological constant. Then Equation (%
\ref{F12}) infers%
\begin{equation}
T_{ij}=pg_{ij}.  \label{F17}
\end{equation}

For a $\mathcal{W}^{\star }-$flat space-time, the scalar curvature is
constant. Thus $\mu =\mathrm{constant}$ and consequently $p=\mathrm{constant}
$. Therefore, the covariant derivative of equation (\ref{F17}) implies $%
\nabla _{l}T_{ij}=0.$

\begin{theorem}
Let $M$ be a perfect fluid $\mathcal{W}^{\star }-$flat space-time obeying
Equation (\ref{4-F}), then the $\mu $ and $p$ are constants and $\mu +p=0$
i.e. the perfect fluid performs as a cosmological constant. Moreover, $%
\nabla _{l}T_{ij}=0$.
\end{theorem}

The following results are two direct consequences of being $\mathcal{W}%
^{\star }-$curvature flat.

\begin{corollary}
A $\mathcal{W}^{\star }-$flat space-time $M$ obeying Equation (\ref{F2-30})
is a Euclidean space.
\end{corollary}

\begin{corollary}
Let $M$ be a dust fluid $\mathcal{W}^{\star }-$flat space-time satisfying
Equation (\ref{4-F}) (i.e. $T_{ij}=\mu u_{i}u_{j}$). Then $M$ is a vacuum
space-time(i.e. $T_{ij}=0$).
\end{corollary}


\begin{thebibliography}{99}
\bibitem{Ahsan:2017} Zafar Ahsan and Musavvir Ali, \emph{Curvature Tensor
for the Spacetime of General Relativity}, Int. J. Geom.Meth. Mod. Phys.
14(5), 1750078(2017).

\bibitem{De:2014} De, U. C., and Ljubica Velimirovic \emph{Spacetimes with
Semisymmetric Energy-Momentum Tensor'} International Journal of Theoretical
Physics 54, no. 6 (June 2015): 1779--83.

\bibitem{Krupka:1996} D. Krupka, \emph{The trace decomposition of tensors of
type }$\left( \emph{1,2}\right) $\emph{\ and }$\left( \emph{1,3}\right) $,
In: L. Tamassy and J. Szenthe (eds.), New Developments in Differential
Geometry, 243-253(1996).

\bibitem{Mallick:2014} S. Mallick and U. C. De, \emph{Space-times admitting }%
$\mathcal{W}_{\emph{2}}$\emph{-curvature tensor}, Int. J. Geom.Meth. Mod.
Phys. 11(4), 145003(2014) .

\bibitem{Mallick:2016} Sahanous Mallick; Young Jin Suh; Uday Chand De, \emph{%
A spacetime with pseudo-projective curvature tensor}, Journal of
Mathematical Physics, 57,062501(2016).

\bibitem{Mantica:2012} Carlo Alberto Mantica and Young Jin Suh, \emph{Pseudo
Z Symmetric Riemannian Manifolds with Harmonic Curvature Tensors},
International Journal of Geometric Methods In Modern Physics, Vol. 9, No. 1
(2012) 1250004 (21 pages.)

\bibitem{Mantica:2014} Carlo Alberto Mantica and Young Jin Suh, \emph{Pseudo
Z Symmetric Space-times}, Journal of Mathematical Physics 55, 042502 (2014).

\bibitem{Mirzoyan:1991} V. A. Mirzoyan, \emph{Ricci semisymmetric
submanifolds }(Russian), Itogi Nauki i Tekhniki. Ser. Probl. Geom.23, 29--66
(1991). VINITI, Moscow

\bibitem{Mishra:1984} R. S. Mishra, \emph{Structures on Differentiable
Manifold and Their Applications}, Chandrama Prakasana, Allahabad, 1984.

\bibitem{ONeill:1983} B. O'Neill, \emph{Semi-Riemannian Geometry, }Academic
Press, New York, 1983.

\bibitem{Ozen:2011} F. Ozen Zengin, \emph{On Riemannian manifolds admitting
W2-curvature tensor}, Miskolc Math. Notes, 12, 289--296(2011).

\bibitem{Pokhariyal:1970} G. P. Pokhariyal and R. S. Mishra, \emph{The
curvature tensor and their relativistic significance}, Yokohama Math. J. 18,
105--108(1970).

\bibitem{Pokhariyal:1971} G. P. Pokhariyal and R. S. Mishra, \emph{Curvature
tensor and their relativistic significance II}, Yokohama Math. J. 19,
97--103(1971).

\bibitem{Pokhariyal:1972} G. P. Pokhariyal, \emph{Curvature tensor and their
relativistic significance III}, Yokohama Math. J. 20, 115--119(1972).

\bibitem{Pokhariyal:1982} G. P. Pokhariyal, \emph{Relative significance of
curvature tensors}, Int. J. Math. Math. Sci. 5, 133--139(1982).

\bibitem{Pokhariyal:2001} G. P. Pokhariyal, \emph{Curvature tensors on
A-Einstein Sasakian manifolds}, Balkan J. Geom. Appl. 6, 45--50(2001).

\bibitem{sach:1977} Sach, R. K. and Hu, W., \emph{General Relativity for
Mathematician}, Springer Verlag, New York, 1977.

\bibitem{Shenawy:2016} S. Shenawy and Bulent Unal, \emph{The }$\mathcal{W}%
_{2}-$\emph{curvature tensor on warped product manifolds and applications},
Int. J. Geom.Meth. Mod. Phys. 13(7), 1650099(2016).

\bibitem{Szabo:2016} Z. I. Szabo,\emph{Structure theorems on Riemannian
spaces satisfying }$R(X,Y)\cdot R=0$, J. Diff. Geom. 17, 531--582 (1982).

\bibitem{Taleshian:2010} A. Taleshian and A. A. Hosseinzadeh, \emph{On }$%
\mathcal{W}_{\emph{2}}$\emph{-curvature tensor of }$\emph{N(k)-}$\emph{%
quasi-Einstein manifolds}, J. Math. Comput. Sci. 1(1) 28--32(2010).
\end{thebibliography}
\end{document}